\documentclass[12pt]{article}
\usepackage{caption}
\usepackage{diagbox}
\usepackage{verbatim}
\usepackage[usenames]{color}
\usepackage{amssymb}
\usepackage{graphicx}
\usepackage{mathtools}
\usepackage{amscd}
\usepackage{algorithm}
\usepackage{enumerate}
\usepackage{cases}
\usepackage[noend]{algpseudocode}
\usepackage[shortlabels]{enumitem}
\usepackage{array, booktabs, multirow}
\usepackage{enumitem}
\usepackage{multicol}
\usepackage{moresize}

\usepackage{tikz}
\usetikzlibrary{decorations.text}

\usepackage{listings}

\definecolor{verbgray}{gray}{0.9}

\lstnewenvironment{code}{%
  \lstset{backgroundcolor=\color{verbgray},
 frame=single,
  language=C,
  framerule=0pt,
  basicstyle=\ttfamily,
  showstringspaces=false,
  commentstyle=\color{blue}\textit,
  keywordstyle=\color{black}\bf,
  numbers=none,
  keepspaces=true,
  columns=fullflexible}}{}

\usepackage[colorlinks=true,
linkcolor=webgreen,
filecolor=webbrown,
citecolor=webgreen]{hyperref}

\definecolor{webgreen}{rgb}{0,.5,0}
\definecolor{webbrown}{rgb}{.6,0,0}

\usepackage{xcolor}
\usepackage{color}
\usepackage{fullpage}
\usepackage{float}

\usepackage{graphics,amsmath,amssymb}
\usepackage{amsthm}
\usepackage{amsfonts}
\usepackage{latexsym}
\usepackage{epsf}

\definecolor{shadecolor}{rgb}{.91, .91, .91}
\definecolor{darkshade}{rgb}{0.7, 0.7, 0.7}

\definecolor{bordercolor}{rgb}{.99, .5, 0}
\definecolor{bordercolor2}{rgb}{0, 0.125, 0.376}
\definecolor{bordercolor3}{rgb}{0.07, 0.53, 0.03}

\definecolor{ultramarine}{rgb}{0, 0.125, 0.576}

 \definecolor{arsenic}{rgb}{0.23, 0.27, 0.29}
 \definecolor{beige}{rgb}{0.96, 0.96, 0.86}

\definecolor{amber}{rgb}{1.0, 0.75, 0.0}
\definecolor{orange}{rgb}{1.0, 0.49, 0.0}
\definecolor{dandelion}{rgb}{0.94, 0.88, 0.19}

  \definecolor{indiagreen}{rgb}{0.07, 0.53, 0.03}
  \definecolor{huntergreen}{rgb}{0.21, 0.37, 0.23}

\definecolor{brightcyan}{RGB}{0, 200, 255}
\newcommand{\blue}[1]{\textcolor{brightcyan}{#1}}
\newcommand{\red}[1] {\textcolor{red}{#1}}

\newcommand{\defo}[1] {{\bf \textcolor{ultramarine}{#1}}}

\newcommand{\seqnum}[1]{\href{https://oeis.org/#1}{\underline{#1}}}

\DeclareMathOperator{\suf}{suf}
\DeclareMathOperator{\runs}{runs}
\DeclareMathOperator{\Ns}{\textbf{N}}
\DeclareMathOperator{\Nn}{N}
\DeclareMathOperator{\Ls}{\textbf{L}}
\DeclareMathOperator{\Ln}{L}
\DeclareMathOperator{\Qs}{\textbf{Q}}
\DeclareMathOperator{\Qn}{Q}
\DeclareMathOperator{\Ps}{\textbf{P}}

\DeclareMathOperator{\PSuf}{psuf}
\DeclareMathOperator{\Next}{MS}
\DeclareMathOperator{\inc}{inc}
\DeclareMathOperator{\dec}{dec}
\DeclareMathOperator{\decF}{decFirst}

\newsavebox{\mybox}

\newcommand{\xeq}[1]{\overset{#1}{=}}

\usepackage{breakurl}

\makeatletter
\newcommand*{\inlineequation}[2][]{%
  \begingroup
    \refstepcounter{equation}%
    \ifx\\#1\\%
    \else
      \label{#1}%
    \fi
    \relpenalty=10000 %
    \binoppenalty=10000 %
    \ensuremath{%
      #2%
    }%
    ~\@eqnnum
  \endgroup
}
\makeatother
\usepackage{hyperref}

\begin{document}

\theoremstyle{plain}
\newtheorem{theorem}{Theorem}
\newtheorem{corollary}[theorem]{Corollary}
\newtheorem{lemma}[theorem]{Lemma}
\newtheorem{proposition}[theorem]{Proposition}
\newtheorem{remark}[theorem]{Remark}

\theoremstyle{definition}
\newtheorem{definition}[theorem]{Definition}
\newtheorem{example}[theorem]{Example}
\newtheorem{conjecture}[theorem]{Conjecture}
\MakeRobust{\Call}

\title{Necklaces and Lyndon words in colexicographic order}

\author{Daniel Gabri\'{c} and Joe Sawada \\
School of Computer Science \\
University of Guelph \\
50 Stone Road East \\
Guelph, ON  N1G 2W1 \\
Canada\\
\href{mailto:dgabric@uoguelph.ca}{\tt dgabric@uoguelph.ca} \\
\href{mailto:jsawada@uoguelph.ca}{\tt jsawada@uoguelph.ca}
}
\date{}

\maketitle

\begin{abstract}
We present the first constant-amortized-time algorithms for generating all length-$n$ necklaces and Lyndon words over a $k$-letter alphabet in colexicographic order, for arbitrary $k\geq 2$. Our approach introduces a novel class of words called \emph{quasinecklaces},  which serve as an easily generated superset of necklaces through which all necklaces can be efficiently identified. We derive a formula for the number $\Qn_k(n)$ of length-$n$ quasinecklaces and show that $\Qn_k(n)$ is proportional to the number of length-$n$ necklaces, which is the key property needed to achieve constant amortized time. We also apply our results to efficiently generate a well-known de Bruijn sequence and efficiently generate necklaces and Lyndon words subject to a weight constraint.
\end{abstract}


\section{Introduction}

Let $\Sigma_k$ denote the $k$-letter totally ordered alphabet $\{1,2,\ldots, k\}$ where $1<2<\cdots < k$. Throughout this paper, we implicitly assume all words have symbols from $\Sigma_k$. We call a word $w$ a \emph{necklace}\footnote{Necklaces are also commonly referred to as equivalence classes of words under rotation.} if it is, not necessarily uniquely, lexicographically smallest among all of its nontrivial rotations. The word $w$ is a \emph{Lyndon word} if it is strictly smaller than all of its nontrivial rotations. For example, the word $113213$ is both a necklace and a Lyndon word, $121212$ is a necklace but not a Lyndon word, and the word $12312$ is neither a Lyndon word nor a necklace. Necklaces and Lyndon words are classical objects in combinatorics on words and have many practical applications. For example, they appear in string algorithms~\cite{Bannai&etal:2017,Duval:1983}, coding theory~\cite{Golomb:1967,Golomb&Gordon&Welch:1958}, bionformatics~\cite{Bonomo&etal:2014,Martayan&etal:2024,Olbrich:2025}, and even music theory~\cite{Chemillier:2004}.

The study of necklaces and Lyndon words naturally raises the question of how to generate them efficiently. We call a generation algorithm constant amortized time (CAT) if the total work performed divided by the number of objects generated is bounded by a constant. Several CAT algorithms for generating necklaces and Lyndon words in lexicographic order are known. Fredricksen, Kessler, and Maiorana~\cite{Fredricksen&Kessler:1986,Fredricksen&Maiorana:1978} presented the FKM algorithm, which generates necklaces in lexicographic order and was subsequently shown to be CAT by Ruskey et al.~\cite{Ruskey&etal:1992}. Duval~\cite{Duval:1988} independently discovered a different lexicographic algorithm for necklaces, later proved to run in CAT by Berstel and Pocchiola~\cite{Berstel&Pocchiola:1994}. Notably, the lexicographically least rotation of a de Bruijn sequence can be obtained by concatenating the Lyndon words generated by Duval's algorithm.

Although necklaces and Lyndon words can already be efficiently generated in lexicographic order, some applications rely on different orderings. For example, the Grandmama de Bruijn sequence~\cite{Dragon&etal:2018} is constructed by concatenating the primitive roots of all necklaces arranged in colexicographic (colex) order, and colex orderings of necklaces also arise in the construction of universal cycles for subsets~\cite{Campbell&etal:2026}. Sawada et al.~\cite{Sawada&Williams:2013,Sawada&Williams&Wong:2017} introduced a CAT algorithm to generate binary necklaces in colex order. Their method relies on an auxiliary set of words called pseudonecklaces, which are easy to generate in colex order and whose cardinality is proportional to that of binary necklaces. For alphabets of size $k\geq 3$, however, no such CAT generation algorithm has been available.


\textbf{Main results.} In this paper, we resolve this gap by presenting the first CAT algorithms for generating $k$-ary necklaces and Lyndon words in colex order for arbitrary $k\geq 2$. To do so, we introduce a novel class of words called quasinecklaces. Quasinecklaces are similar in spirit to pseudonecklaces: every pseudonecklace is a quasinecklace, but the converse does not hold in general. We show that the number $\Qn_k(n)$ of length-$n$ quasinecklaces over $\Sigma_k$ is $\mathcal{O}(\Nn_k(n))$, where $\Nn_k(n)$ is the number of length-$n$ necklaces over $\Sigma_k$, which is the key property needed to achieve constant amortized time. This bound also holds if we add an upper bound on the weight of the strings. We apply these results to efficiently generate the Grandmama de Bruijn sequence and bounded-weight universal cycles.  We provide a complete C implementation in Appendix~\ref{appendix:neck}.

\textbf{Outline.} The remainder of the paper is organized as follows. Section~\ref{section:preliminaries} establishes preliminary definitions and notation. Section~\ref{section:quasi} introduces quasinecklaces, presents a CAT algorithm for generating them in colex order, and derives a formula for their count $\Qn_k(n)$. Section~\ref{section:neck} describes how to modify the quasinecklace algorithm to generate necklaces or Lyndon words in colex order, analyzes its running time, and extends the method to handle a weight constraint. In Section~\ref{section:app}, we apply our results to efficiently a (bounded-weight) de Bruijn sequence. We conclude with a summary and open questions in Section~\ref{section:open}. In this paper, we assume the unit-cost RAM model of computation.

\section{Preliminaries}\label{section:preliminaries}

Recall that $\Sigma_k$ denotes the $k$-letter totally ordered alphabet $\{1,2,\ldots, k\}$ where $1<2<\cdots < k$. Throughout this paper, we implicitly assume all words have symbols from $\Sigma_k$ and all words are $1$-indexed. Let $w$ be a word of length $n$. Let $i$ and $j$ be integers such that $1\leq i,j \leq n$. Let $\min(u)$ denote the smallest symbol of $u$. We denote the $i$'th symbol of $w$ as $w[i]$. We say that $y$ is a \emph{factor} of $w$ if there exist (possibly empty) words $x$ and $z$ such that $w=xyz$. We will use the notation $w[i:j]$ to denote the factor of $w$ starting at index $i$ and ending at index $j$ where $w[i:j]=\epsilon$ if $i>j$. Any factor of the form $w[1:i]$ (resp. $w[i:n]$) is called a \emph{prefix} (resp. \emph{suffix}). If $i<n$ (resp. $i>1$), then we call the prefix (resp. suffix) \emph{proper}. For $p>0$, let $w^p = ww^{p-1}$ where $w^0 = \epsilon$.  The word $w$ is said to be a \emph{power} if $w=u^p$ for some word $u$ and some integer $p>1$; otherwise, we say that $w$ is \emph{primitive}. The \emph{primitive root} of $w$ is the shortest prefix $u$ of $w$ such that $w=u^p$ for some $p\geq 1$.

Let $u$ and $v$ be words of lengths $m$ and $n$ respectively. Then $u$ comes before $v$ in \emph{lexicographic} order if either $u$ is a prefix of $v$ or if there exists an integer $i$ such that $u[1:i-1] = v[1:i-1]$ and $u[i]< v[i]$. Additionally, $u$ comes before $v$ in \emph{colexicographic} (colex) order if either $u$ is a suffix of $v$ or if there exists an integer $i$ such that $u[m-i+1:m] = v[n-i+1:n]$ and $u[m-i] < v[n-i]$. If $u$ is lexicographically less than (or equal to) $v$, we write $u<v$ (resp. $u \leq v$). 

Let $u$ and $v$ be words of length $n$. We say that $u$ is a \emph{rotation} of $v$ if there exist words $x$, $y$ such that $u=xy$ and $v=yx$. We formalize this notion by introducing the left-shift map $\sigma$ so that $\sigma^i(u) = yx$ where $u=xy$ and $|x|=i$, where $i$ is an integer such that $0\leq i \leq |u|$. We say that $u$ is a \emph{nontrivial} rotation of $v$ if $\sigma^i(u)=v$ where $i$ is an integer such that $0 < i < |u|$.

We can now formally define necklaces and Lyndon words. The word $w$ is said to be a \emph{necklace} if $w \leq \sigma^i(w)$ for all integers $i$ such that $0 < i < |w|$. The word $w$ is said to be a \emph{Lyndon word} if $w < \sigma^i(w)$ for all integers $i$ such that $0< i < |w|$. Equivalently, the word $w$ is a Lyndon word if and only if it is lexicographically smaller than all of its proper suffixes. Note that every Lyndon word is primitive and the primitive root of a necklace is a Lyndon word. Let $\Ns_k(n)$ (resp. $\Ls_k(n)$) denote the set of length-$n$ necklaces (resp. Lyndon words) over $\Sigma_k$ and let $\Nn_k(n)=|\Ns_k(n)|$ (resp. $\Ln_k(n)=|\Ls_k(n)|$). See sequences \seqnum{A000031}, \seqnum{A001037}, \seqnum{A001867}, \seqnum{A027376}, \seqnum{A001868}, and \seqnum{A027377} in the \emph{On-Line Encyclopedia of Integer Sequences}~\cite{OEIS} for sample values of $\Nn_2(n)$, $\Ln_2(n)$, $\Nn_3(n)$, $\Ln_3(n)$, $\Nn_4(n)$, and $\Ln_4(n)$, respectively.  The following two remarks are proved in~\cite{Cattell:2000}.

\begin{remark} \label{count:lyn}
    For $n,k \geq 2$, $\Ln_k(n) \in \Theta(\Nn_k(n))$.
\end{remark}
\begin{remark} \label{count:sum}
    For $n \geq 1$ and $k \geq 2$, $\sum\limits_{t=1}^n \Nn_k(t) \in \Theta(\Nn_k(n))$.
\end{remark}

\begin{definition}\label{def:pseudo}
Let $w$ be a word of length $n$, $a = \min(w)$, and $\ell$ be the length of a largest run of $a$'s in $w$. We say that $w$ is a \emph{quasinecklace} if the following conditions are satisfied:
    \begin{enumerate}[label=(\alph*)]
        \item The word $a^\ell$ is a prefix of $w$.
        \item If $w[i+1:i+\ell]=a^\ell$ for any $i$ where $1\leq i+\ell< n$, then $w[\ell+1]\leq w[i+\ell+1]$.
        \item If $w[n]=a$, then $w=a^n$.
    \end{enumerate}
\end{definition}

\noindent Let $\Qs_k(n)$ denote the set of length-$n$ quasinecklaces and let $\Qn_k(n)=|\Qs_k(n)|$. See Table~\ref{table:example} in Appendix~\ref{appendix:table} for a list of Lyndon words, necklaces, and quasinecklaces for $n=5$, $k=3$ as the appear in colex order. 

A binary word is a \emph{pseudonecklace} if its maximal prefix of the form $1^*2^*$ is lexicographically no larger than any other maximal factor of the form $1^*2^*$. Let $\Ps(n)$ denote the set of all length-$n$ pseudonecklaces. 
See Table~\ref{table:pseudo-quasi} in Appendix~\ref{appendix:pseudo} for a list of quasinecklaces and pseudonecklaces for $n=8$, $k=2$.

\section{Generating Quasinecklaces}\label{section:quasi}

In this section, we outline the approach we take to efficiently generating quasinecklaces in colex order.

The key step in the approach of Sawada et al.~\cite{Sawada&Williams&Wong:2017} to efficiently generate all \emph{binary} necklaces was to identify the set $\Ps(n)$ of pseudonecklaces. This step is important because pseudonecklaces can easily be generated in colex order and the number of pseudonecklaces is proportional to the number of binary necklaces (i.e., $|\Ps(n)|\in \mathcal{O}(\Nn_2(n))$). Given these two facts, generating binary necklaces efficiently requires only a simple and ``cheap'' test to rule out any pseudonecklace that is not a binary necklace. We follow the same approach but use a different, and slightly larger, set of length-$n$ words called quasinecklaces.

Quasinecklaces are very similar to pseudonecklaces. In fact, it is not too hard to show that every pseudonecklace is a quasinecklace and that there are quasinecklaces that are not pseudonecklaces. 
Our first step is to show that the number of quasinecklaces is proportional to the number of necklaces.
\begin{theorem}\label{theorem:QNeck}
Let $k\geq 2$ and $n\geq 1$ be integers.  Then $\Qn_k(n) \in \mathcal{O}(\Nn_k(n))$.
\end{theorem}
\begin{proof}
    It is sufficient to show that there is a one-to-one map from $\Qs_k(n) - \Ns_k(n)$ to $\Ns_k(n)$. Let $w\in \Qs_k(n) - \Ns_k(n)$. Since $w$ is not a necklace, we have that $w\neq a^n$ for all $a\in \Sigma_k$. Therefore, we can write $w=a^\ell b u$ where $a< b$ and $\ell\geq 1$. We define the mapping $w \mapsto w'$ where $w'= a^\ell (b-1) u$. We prove that the mapping is one-to-one, since it is clear that $w'$ is a necklace. 
    
     Suppose $w_1,w_2\in \Qs_k(n) - \Ns_k(n)$ and $w_1\neq w_2$. We can write $w_1=a_1^{\ell_1}b_1 u_1$  and $w_2=a_2^{\ell_2}b_2 u_2$  where $a_1<b_1$, $a_2<b_2$, and $\ell_1,\ell_2\geq 1$. 
    For the sake of a contradiction, suppose $a_1^{\ell_1}(b_1-1) u_1=a_2^{\ell_2}(b_2-1) u_2$. If $\ell_1=\ell_2$, then $w_1=w_2$, a contradiction. Assume $\ell_1>\ell_2$. Since $w_1$ is not a necklace, there is a proper suffix $v$ of $w_1$ that is smaller than $w_1$. This suffix $v$ must begin with $a_1^{\ell_1}b_1$ and must be a suffix of $u_1$. But since $\ell_1>\ell_2$ and $a_1^{\ell_1}(b_1-1) u_1=a_2^{\ell_2}(b_2-1) u_2$, we have that $u_1$ is a suffix of $u_2$. This implies that $a_1^{\ell_1}=a_2^{\ell_1}$ is a factor of $u_2$. But this contradicts $w_2$ being a quasinecklace. Therefore, the mapping is one-to-one.
\end{proof}

Before we present our algorithm to generate quasinecklaces, we introduce two functions, the second of which is central to our algorithm.  Let $\runs(u)=(a,\ell,b,\ell_1)$ such that
\begin{itemize}
    \item $a=\min(u)$,
    \item $\ell$ is the maximum number of contiguous $a$'s in $u$, 
    \item $b$ is the smallest symbol that follows $a^\ell$ in $u$ where $b = a$ if $a^\ell$ is a suffix of $u$, and
    \item $\ell_1$ is the maximum number of contiguous $a$'s in the prefix of $u$.
\end{itemize} 
Let $\Next(j,u, \runs(u))$ be the largest symbol $c\in \Sigma_k$ such that $cu$ is a suffix of some quasinecklace of length $n=j+|u|$. When expanding $\runs(u)$ in $\Next(j,u, \runs(u))$, we abuse notation and write $\Next(j,u,a,\ell,b,\ell_1)$. For example, if $k=3$, then $\Next(112,3,\runs(112))=\Next(112,3,1,2,2,2)=2$ since $3112$ is not the suffix of a length-$6$ quasinecklace, but $2112$ is. In our recursive algorithm, we maintain a suffix of some quasinecklace in $\Qs_k(n)$ and use $\Next$ to compute the possible symbols we could use to extend this suffix. In Lemma~\ref{lemma:branch}, we fully characterize $\Next(j,u,\runs(u))$.

\begin{remark}\label{remark:pad}
    If $u$ is a suffix of some quasinecklace in $\Qs_k(n)$, then $1^{n-|u|-1}cu$ is a quasinecklace in $\Qs_k(n)$ for all symbols $c$ where $1\leq c \leq \Next(j,u,\runs(u))$. 
\end{remark}

\begin{lemma}\label{lemma:branch}
    Let $n$ and $j$ be integers such that $1 \leq j \leq n$. Let $u$ be the length-$(n-j)$ suffix of some quasinecklace $w\in \Qs_k(n)$. Let $(a,\ell,b,\ell_1)=\runs(u)$. Let $b_1 = u[\ell_1+1]=w[j+\ell_1+1]$ if $\ell_1+1 \leq n-j$; otherwise let $b_1=u[1]=w[j+1]$.  
    Then
    {\footnotesize
    \begin{numcases}{\Next(j,u,\runs(u)) = } 
      1 , & \text{if $u$ ends with $1$;} \label{case1}\\ 
      k , & \text{if $j=n$ or $j>\ell +1$ or both $j>1$ and $a>1$;} \label{case2}\\ 
      b , & \text{if $j=\ell+1$ or $\ell = n-1$;}\label{case3} \\  
      a , & \text{if $j=1$ and $w[n]>a$ and $b_1 \leq b$ and $\ell_1+1= \ell$;} \label{case4} \\
      a , & \text{if $j=1$ and $w[n]>a$ and $\ell_1+1> \ell$;} \label{case5} \\
      \max(a-1,1) , & \text{otherwise.} \label{case7}
   \end{numcases}}
\end{lemma}

\begin{proof} Let $c\in \Sigma_k$ be the largest symbol such that $cu$ is a suffix of a quasinecklace in $\Qs_k(n)$. We prove that $\Next_n(j,u,\runs(u))=c$ for each case.

\begin{enumerate}
    \item[(\ref{case1})] Suppose $u$ ends with $1$. There is exactly one quasinecklace in $\Qs_k(n)$ that ends with $1$ and it is $1^n$. Thus $c=1$.
\end{enumerate}
    For the remaining cases we have that $u$ does not end with $1$.
\begin{enumerate}
    \item[(\ref{case2})] If $j=n$, then $c=k$ since the last symbol of a quasinecklace is constrained. If $j>\ell+1$ or both $j>1$ and $a>1$, then $1^{j-1} ku$ is a quasinecklace, so $c=k$.
\end{enumerate}
    For the remaining cases we have $j<n$, $j\leq \ell+1$, and either $j=1$ or $a=1$.
\begin{enumerate}
    \item[(\ref{case3})]
    \begin{itemize}
    \item    Suppose $j= \ell+1$. First we prove that $w=1^\ell b u$ is a quasinecklace. We must have $b>1$, since $b=1$ implies that $u$ ends with $1$. We immediately get that $w$ satisfies all conditions of Definition~\ref{def:pseudo}. Therefore $w$ is a quasinecklace
    
    Now we show that $b$ is the largest symbol such that $bu$ is a suffix of some quasinecklace in $\Qs_k(n)$. Suppose $d>b$ and consider the word $du$. Since $j=\ell+1>1$ and either $j=1$ or $a=1$, we have that $a=1$. Any quasinecklace in $\Qs_k(n)$ that has $du$ as a suffix must begin with $1^\ell$. Now $w=1^\ell du\in \Qs_k(n)$ by Remark~\ref{remark:pad}. But $1^\ell d$ is a prefix of $w$, which contradicts condition (b), since $d > b$. Therefore $c=b$. 
    
    \item Suppose $\ell = n-1$. Then $u=b^{n-1}$. Clearly $bu=b^n$ is a quasinecklace. It is also clear that $c=b$ is the largest symbol such that $w=cu$ is a quasinecklace. 
    \end{itemize}

\end{enumerate}
    For the remaining cases we have $j<\ell+1$ and $\ell < n-1$. For cases (\ref{case4}) and (\ref{case5}) we have $c\leq a$, since $j=1$ and any quasinecklace in $\Qs_k(n)$ that has $u$ as a suffix must not begin with a symbol larger than $a$. Thus, to show that $c=a$ for the following two cases, it is sufficient show that $w=au$ is a quasinecklace.
\begin{enumerate}
    \item[(\ref{case4})] Suppose $j=1$, $w[n]>a$, $b_1 \leq b$, and $\ell_1+1= \ell$. Then $w$ begins with $a^\ell b_1$. But this means that $w$ satisfies condition (a) of Definition~\ref{def:pseudo}. Condition (b) is satisfied because $b_1\leq b$. Condition (c) is satisfied since $w[n] \neq a$. Therefore $w$ is a quasinecklace. 
\end{enumerate}
For the remaining cases we have $j>1$,$w[n]\leq a$, $b_1 > b$, or $\ell_1+1\neq \ell$.
\begin{enumerate}    
    \item[(\ref{case5})] Suppose $j=1$, $w[n]>a$, and $\ell_1+1> \ell$. Since $\ell_1+1 > \ell$, $w$ begins with at least $\ell+1$ copies of $a$, the unique longest run of $a$'s in $w$. Therefore, $w$ satisfies conditions (a) and (b) of Definition~\ref{def:pseudo}. Condition (c) is also satisfied because $w[n]>a$. Thus, $w$ is a quasinecklace.
\end{enumerate}
For the remaining cases we have $j>1$, $w[n]\leq a$, or $\ell_1+1\leq \ell$.
\begin{enumerate} 
    \item[(\ref{case7})] Suppose all the conditions in cases (\ref{case1})-(\ref{case5}) do not hold. We show that $c=\max(a-1,1)$ by showing that if $a=1$, then $c=1$, and if $a>1$, then $c=a-1$.

    \begin{itemize}
        \item Suppose $a=1$. From the end of the proof of case (\ref{case3}), we have $j<\ell +1$, which implies $c\leq a$.  By Remark~\ref{remark:pad} we get that $1^{n-j}u$ is a quasinecklace. Thus, $c=1$.
        \item Suppose $a>1$. Then we know that $j=1$, since the end of the proof of case (\ref{case2}) shows that $j=1$ or $a=1$. We prove that the word $w=du$ where $d=a-1$ is a quasinecklace. Conditions (a) and (b) of Definition~\ref{def:pseudo} are satisfied since $w$ begins with the unique copy of the smallest symbol in $w$. Condition (c) is satisfied since $\min(u)=a$ and so $w[n]\geq a >a-1$. Therefore, $w$ is a quasinecklace and $c\geq a-1$.

        Now we show that $c\leq a-1$ by showing that $w=du$ is not a quasinecklace for $d\geq a$. If $w[n]=a$, then condition (c) is not satisfied since $\ell < n-1$ and so $w\neq a^n$. So $w[n]>a$. From the end of the proof of case (\ref{case5}) we have that either $\ell_1+1 = \ell$ or $\ell_1 + 1 < \ell$. If $\ell_1+1=\ell$, then we must have $b_1>b$ by the end of case (\ref{case4}). But now we have that $a^\ell b_1$ is a prefix of $w$, which violates condition (b). So suppose that $\ell_1 + 1 < \ell$. Now we have that $w$ begins with fewer than $\ell$ copies $a$'s, which violates condition (a). Therefore, we have $c\leq a-1$. 
    \end{itemize}
\end{enumerate}
\end{proof}

The value $\Next(j,u,\runs(u))$ can clearly be computed in constant time. The function \Call{Gen}{} in Algorithm~\ref{algo:QN} generates the quasinecklaces $\Qs_k(n)$  in colex order. The function $\Call{Visit}{}$ processes, or outputs, the current word $w$. The initial call is $\Call{Gen}{n,0,0,0,0}$.

\begin{algorithm}[H] 
\footnotesize
\caption{The function \Call{Gen}{} generates all  quasinecklaces in $\Qs_k(n)$ in colex order.}\label{algo:QN}
\begin{algorithmic}[1]
\Function{Gen}{$j,a,\ell,b,\ell_1$}
    \State $MaxSymbol \gets \Next(j,w[j+1:n],a,\ell,b,\ell_1)$
    \State $b_1 \gets w[j+\ell_1+1]$ \textbf{if} {$\ell_1 +1 \leq n-j$} \textbf{else} $a$ 
    \If{$MaxSymbol = 1$ \textbf{or} $j = 0$}   \Call{Visit}{ } 
    \Else
        \For{$c\gets 1$ \textbf{to} $MaxSymbol$}
            \State $w[j]\gets c$
            \If{$j=n$} $\Call{Gen}{n-1,c,1,c,1}$
            \ElsIf{$c < a$} $\Call{Gen}{j-1,c,1,w[j+1],1}$
            \ElsIf{$c=a$ \textbf{and} ($\ell_1+1 > \ell$ \textbf{or} $\ell_1+1 = \ell$ \text{and} $b_1 \leq b$)} $\Call{Gen}{j-1,a,\ell_1+1,b_1,\ell_1+1}$
            \ElsIf{$c=a$} $\Call{Gen}{j-1,a,\ell,b,\ell_1+1}$
            \Else \ $\Call{Gen}{j-1,a,\ell,b,0}$
            \EndIf

        \EndFor
        \State $w[j]\gets 1$
    \EndIf
\EndFunction
\end{algorithmic}
\end{algorithm}

We briefly justify that this algorithm generates each quasinecklace in constant-amortized time. Throughout this proof sketch we refer to the recursion tree of the algorithm where each node represents a suffix of some quasinecklace in $\Qs_k(n)$, and each leaf is a quasinecklace in $\Qs_k(n)$. Simply put, a generation algorithm runs in constant-amortized time if the total amount of work done divided by the number of objects generated is bounded by a constant. Therefore, it is sufficient to show that the total number of nodes in the recursion tree is proportional to the total number of leaf nodes and that the amount of work done by each node is bounded by a constant. Let $MaxSymbol=\Next(j,w[j+1:n],\runs(w[j+1:n]))$. When $MaxSymbol=1$, the algorithm immediately stops and processes $w$, as there is only one quasinecklace with suffix $w[j+1:n]$, namely $1^jw[j+1:n]$. We avoid any additional work involved with filling in $1^j$ by initializing $w$ to $1^n$ and maintaining the prefix to be $1^j$ after all recursive calls on Line 13. When $MaxSymbol>1$, the algorithm performs $MaxSymbol$ recursive calls to add $MaxSymbol$ different symbols to the beginning of $w[j+1:n]$. Therefore, each internal node of the recursion tree has at least two children, which implies that the total number of nodes is proportional to the total number of leaves. By inspection, it is clear that every node in the recursion tree is responsible for a constant amount of work.  Since the depth of the recursion is at most $n$, the algorithm uses $\mathcal{O}(n)$ space.
Thus, we get the following theorem.

\begin{theorem}\label{theorem:quasi}
Algorithm~\ref{algo:QN} generates the quasinecklaces $\Qs_k(n)$ in colex order in constant amortized time using $\mathcal{O}(n)$ space.
\end{theorem}

It is important to note that Theorem~\ref{theorem:quasi} assumes the algorithm is not printing out the quasinecklaces, as printing a word $w$ is inherently $\mathcal{O}(|w|)$.

\subsection{Enumeration} \label{subsection:enumeration}
In this section, we provide a simple recurrence for the number $\Qn_k(n)$ of quasinecklaces. First we define a quantity that we use to count $\Qn_k(n)$. Let $A_{a,\ell,b}(n)$ be the number of quasinecklaces in $\Qs_k(n)$ that begin with $a^\ell b$ where $a$, $b$ are symbols such that $a<b$.

\begin{lemma}\label{lemma:quasiCount}
        Let $n$, $\ell$ be positive integers $a$, $b$ be symbols such that $a<b$. Then \[A_{a,\ell, b}(n) = \begin{cases} 
      0 , & \text{if $n\leq \ell$;} \\
      1 , & \text{if $n=\ell+1$;} \\
      (k-a)\sum\limits_{i=1}^{\ell}A_{a,\ell,b}(n-i) + (k-b+1)A_{a,\ell,b}(n-\ell-1) , & \text{otherwise.} \\
   \end{cases}\]
\end{lemma}
\begin{proof}
    Let $w$ be a quasinecklace in $\Qs_k(n)$ beginning with $a^\ell b$. If $n\leq \ell$, then any length-$n$ word cannot begin with $a^\ell b$, thus $A_{a,\ell,b}(n)$. If $n=\ell+1$, then there is exactly one quasinecklace in $\Qs_k(n)$  that begins with $a^\ell b$, namely $w=a^\ell b$, thus $A_{a,\ell,b}(n)=1$. 
    
    Suppose $n>\ell+1$. Write $w=a^\ell u$ for some word $u$ that begins with $b$. We look at the suffixes of $u$.  Since $w$ cannot have the factors $a^{\ell+1}$ or $a^\ell b'$ for some $b' < b$, we have that either $u=vca^\ell d$ where $v$ is a word and $c$, $d$ are symbols such that $c>a$ and $d>b$ or $u=v'ca^{i-1} d'$ where $v'$ is a word, $i$ is some integer such that $1\leq i \leq \ell$, and $d>a$ is a symbol. But now both $a^\ell vc$ and $a^\ell v' c$ are both quasinecklaces since they are prefixes of the quasinecklace $w$ and $c>a$. There are $A_{a,\ell, b}(n-\ell-1)$ quasinecklaces of the form $a^\ell vc$ and $k-b+1$ choices for the symbol $d$. Therefore, there are $A_{a,\ell,b}(n-\ell-1)$ quasinecklaces in $\Qs_k(n)$ of the form $w=a^\ell vca^\ell d$. Additionally, there are $A_{a,\ell,b}(n-i)$ quasinecklaces of the form $a^\ell v' c$ and $k-a$ choices for the symbol $d'$, Thus, there are $(k-a)A_{a,\ell,b}(n-i)$ quasinecklaces in $\Qs_k(n)$ of the form $w=a^\ell v'ca^\ell d'$. Summing $(k-a)A_{a,\ell,b}(n-i)$ over all possible $i$ in addition to summing $A_{a,\ell,b}(n-\ell-1)$, we get the desired formula for $n>\ell+1$.    
\end{proof}


\begin{theorem}\label{theorem:quasiCount}
  Let $n$ and $k$ be integers such that $n\geq 1$ and $k\geq 2$. Then  \[Q_k(n) =k+ \sum_{a=1}^{k-1}\sum_{b=a+1}^k\sum_{\ell=1}^{n-1} A_{a,\ell,b}(n).\]
\end{theorem}
\begin{proof}
A quasinecklace in $\Qs_k(n)$ is either of the form $a^n$ or begins with the word $a^\ell b$ where $a$ and $b$ are symbols such that $a<b$. There are $k$ quasinecklaces of the form $a^n$. There are $A_{a,\ell,b}(n)$ quasinecklaces that begin with $a^\ell b$. Summing $A_{a,\ell,b}(n)$ over all possible $a$, $\ell$, $b$, in addition to summing $k$, we get the desired formula.
\end{proof}

%

\section{Generating Necklaces and Lyndon words}\label{section:neck}

In this section, we describe how to modify Algorithm~\ref{algo:QN} to generate necklaces and Lyndon words in colex order. Applying an $\mathcal{O}(n)$ test on each quasinecklace to determine if it is a necklace leads to a $\mathcal{O}(n)$-amortized time algorithm. To obtain constant-amortized time, we maintain two additional pieces of information as each quasinecklace is generated: 
\begin{itemize}
    \item $\suf(w)$: the length of the lexicographically smallest suffix of $w$, and 
    \item $\PSuf(w,r)$: the length of the longest suffix of $w$ whose primitive root is of length $r$. 
\end{itemize}
By applying the values $\suf(w)$ and $\PSuf(w,r)$ we can test if a word $w$ is a Lyndon word or necklace in constant time applying the following lemma.

\begin{lemma}\label{lemma:suf}
    Let $w$ be a length-$n$ word where $r=\suf(w)$ and $p=\PSuf(w,r)$. Then 
    \begin{enumerate}
        \item[(1)] $w$ is a Lyndon word if and only if $r=n$, and
        \item[(2)] $w$ is a necklace if and only if $p=n$.
    \end{enumerate}
\end{lemma}

\begin{proof} We prove both statements:
\begin{enumerate}
    \item[(1)] Follows from the definition of a Lyndon word.
    \item[(2)]  Suppose the word $w$ is a necklace. Then $w=u^j$ for some $j\geq 1$ where $u$ is a Lyndon word. The word $v=w[n-r+1:n]$ is a Lyndon word, since it is the smallest suffix of $w$. We must have $u=v$, for otherwise we would have that either $u$ or $v$ are not Lyndon words. Therefore, $w$ is the longest suffix of $w$ whose primitive root is of length $r$, and so $p=n$. The steps of this proof are reversible, so the backwards direction has also been covered.
\end{enumerate}
\end{proof}

Lemma~\ref{lemma:suf} implies that it is enough to know the values of $\suf(w)$ and $\PSuf(w,\suf(w))$ to determine whether $w$ is a Lyndon word or necklace. 
Computing these values for every quasinecklace $w$ generated is computationally expensive.
A better approach is to incrementally update these values for each new suffix generated.
Lemmas~\ref{lemma:sufInc} and \ref{lemma:psufInc} show how to compute $\suf(w)$ and $\PSuf(w,\suf(w))$ incrementally. Corollaries~\ref{corollary:sufInc} and \ref{corollary:psufInc} show how to compute $\suf(w)$ and $\PSuf(w,\suf(w))$ if $w=1^j u$ for some $j\geq 1$ and we only know $\suf(u)$ and $\PSuf(u, \suf(u))$.

\begin{lemma}\label{lemma:sufInc}
Let $n$ and $j$ be integers such that $n\geq j \geq 1$. Let $u=w[j+1:n]$ be the length-$(n-j)$ suffix of a  quasinecklace $w \in \Qs_k(n)$, $r=\suf(u)$, and $c$ be an integer such that $1\leq c \leq \Next(j,u,\runs(u))$. Then
    \[\suf(cu) = \begin{cases} 
      |cu| , & \text{if $cw[j+1:j+r-1] < w[n-r+1:n]$;} \\ 
      r , & \text{otherwise.}
   \end{cases}\]
\end{lemma}
\begin{proof}
    If $cw[j+1:j+r-1] < w[n-r+1:n]$, then $cu$ is the smallest suffix of $cu$ and therefore is primitive, so $\suf(cu)=|cu|$. Otherwise, we have that $cw[j+1:j+r-1] \geq w[n-r+1:n]$, which implies $\suf(cu)=\suf(u)=r$.  
\end{proof}
\noindent
The next result follows from $j$ applications of the previous lemma.  
\begin{corollary}\label{corollary:sufInc}
    Let $n$ and $j$ be integers such that $n\geq j \geq 1$. Let $u=w[j+1:n]$ be the length-$(n-j)$ suffix of a quasinecklace $w \in \Qs_k(n)$, $r=\suf(u)$, and $w' = 1^jw[j+1:r]$. Then \[\suf(1^ju) = \begin{cases} 
      n , & \text{if $w'[1:r] < w[n-r+1:n]$}; \\ 
      r , & \text{otherwise.}
   \end{cases}\]
\end{corollary}

\begin{lemma}\label{lemma:psufInc}
Let $n$ and $j$ be integers such that $n\geq j \geq 1$. Let $u=w[j+1:n]$ be the length-$(n-j)$ suffix of a quasinecklace $w \in \Qs_k(n)$. Let $r=\suf(u)$ and $p=\PSuf(u,\suf(u))$. Let $c$ be an integer such that $1\leq c \leq  \Next(j,u,\runs(u))$. Then
\[\PSuf(cu,r) = \begin{cases} 
      |cu| , & \text{if $cw[j+1:j+r-1]<w[n-r+1:n]$;} \\ 
      |cu| , & \text{if $cw[j+1:j+r-1]=w[n-r+1:n]$ and $|cu|- p = r$;} \\  
      p , & \text{otherwise.} 
   \end{cases}\]
\end{lemma}
\begin{proof}
    If $cw[j+1:j+r-1]<w[n-r+1:n]$, then $cu$ is the smallest suffix of $cu$ and is therefore primitive, so $\PSuf(cu,r)=|cu|$. If $cw[j+1:j+r-1]=w[n-r+1:n]$ and $|cu|- p = r$, then $cw[j+1:j+r-1]=v$ and $w[j+r:n]=v^i$ for some $i\geq 1$ where $v=w[n-r+1:n]$. Since $v$ is a Lyndon word by definition and $cu=v^{i+1}$, we have that $cu$ is the longest suffix of $cu$ whose primitive root is of length $r$. Therefore $\PSuf(cu,r)=|cu|$. Otherwise, we have that $cw[j+1:j+r-1]>w[n-r+1:n]$ or $|cu|-p\neq r$. In the case that $cw[j+1:j+r-1]>w[n-r+1:n]$, we have that the word $v=w[n-r+1:n]$ is not a prefix of $cu$. Therefore, the longest suffixes of $cu$ and $u$ that have a primitive root of length $r$ are the same, and so $\PSuf(cu,r)=p$. If $cw[j+1:j+r-1]=w[n-r+1:n]$ and $|cu|-p\neq r$, then the primitive root of $cu$ is not of length $r$. Thus, the longest suffixes of $cu$ and $u$ that have a primitive root of length $r$ are the same, and so $\PSuf(cu,r)=p$.
\end{proof}
\noindent
The next result follows from $j$ applications of the previous lemma. It is important to note that Corollary~\ref{corollary:psufInc} does not hold when $u=1$, but Lemma~\ref{lemma:psufInc} does. Corollary~\ref{corollary:psufInc} assumes that $u$ is of the form $v'v^i$ where $r=\suf(u)=|v|$ and $p=\PSuf(u,r)=|v^i|$. The corollary's second case tests whether prepending $1^j$ completes exactly one additional copy of $v$ (i.e., is $1^ju=v^{i+1}$?). When $u=1$, prepending $1^j$ adds $j$ new copies of $v$ instead of just one, a scenario the second case cannot detect. Therefore, the corollary falls through to the last case and gives the wrong value.
\begin{corollary}\label{corollary:psufInc}
    Let $n$ and $j$ be integers such that $n\geq j \geq 1$. Let $u=w[j+1:n]\neq 1$ be the length-$(n-j)$ suffix of a quasinecklace $w \in \Qs_k(n)$, $r=\suf(u)$, $p=\PSuf(u,\suf(u))$, and $w'=1^jw[j+1:r]$. Then \[\PSuf(1^ju,r) = \begin{cases} 
      n , & \text{if $w'[1:r]<w[n-r+1:n]$;} \\ 
      n , & \text{if $w'[1:r]=w[n-r+1:n]$ and $n- p = r$;} \\  
      p , & \text{otherwise.} 
   \end{cases}\]
\end{corollary}


 The computation of $\suf(cu)$ and $\PSuf(cu,\suf(cu))$ requires the knowledge of whether $cw[j+1:j+r-1]<w[n-r+1:n]$ and $cw[j+1:j+r-1]=w[n-r+1:n]$. In Algorithm~\ref{algo:cmp}, we describe a function \Call{CMP}{} which has four parameters, $p,\ell,q,\ell_1$, and returns the result of a comparison between $w[p:n-q+p]$ and $w[q:n]$. This function has the following preconditions: 
 \begin{itemize}
 \item $p<q$, 
 \item $w[q:n]$ is the smallest suffix of $w[p+1:n]$, 
 \item $w[p:n]$ begins with $a^{\ell_1}$ where $a$ is a symbol, and
 \item if $q<n$, then $w[q:n]$ begins with $a^{\ell}$.
 \end{itemize}
 
The function \Call{CMP}{} is largely straightforward: it first compares the lengths of the runs of $a$ at the beginning of both words, and then performs a symbol-by-symbol comparison between the two words. However, Lines 2 and 6 require closer examination. In the case that $q=n$, we have that $w[q:n]=a$ but the integer $\ell$ may not be $1$; Line 2 handles this special case separately. If $p+i=q$, then $w[p:n] = uuv$ where $w[q:n]=uv$. Since $w[q:n]$ is the smallest suffix of $w[p+1:n]$, we have $uv < v$, which implies $uuv < uv$ and $w[p:n-q+p] < w[q:n]$. Thus, on Line 6 we immediately return that $w[p:n-q+p] < w[q:n]$ when $p+i = q$.

\begin{algorithm}[H] 
\small
\caption{Returns $-1$ if $w[p:n-q+p] <w[q:n]$, $0$ if $w[p:n-q+p]=w[q:n]$, and $1$ if $w[p:n-q+p] > w[q:n]$.} 

\label{algo:cmp}
\begin{algorithmic}[1]
\Function{CMP}{$p,\ell_1,q,\ell$}
    \If{$q = n$} \Return $0$ \EndIf
    \If{$\ell_1 < \ell$} \Return $1$ \EndIf
    \If{$\ell_1 > \ell$} \Return $-1$ \EndIf
    \For{$i\gets \ell$ \textbf{to} $n-q$}
        \If{$p+i = q$} \Return $-1$
        \EndIf
        \If{$w[p+i] > w[q+i]$} \Return $1$
        \ElsIf{$w[p+i] < w[q+i]$} \Return $-1$
        \EndIf
    \EndFor
    \State \Return $0$
\EndFunction
\end{algorithmic}
\end{algorithm}

\subsection{Modifying Algorithm~\ref{algo:QN}}
\label{sec:modify}

Assume that $w$, $j$, $c$, $a$, $\ell$, $b$, and $\ell_1$ are as in Algorithm~\ref{algo:QN}.
The following modifications can be applied to this algorithm to generate necklaces and Lyndon words: 
\begin{enumerate}
    \item Add two new parameters $r$ and $p$ to $\Call{Gen}{}$ such that  $r = \suf(w[j+1:n])$ and $p=\PSuf(w[j+1:n],r)$. 
    \item Update $r$ and then $p$ before each recursive call. 
    \item Before calling the function \Call{Visit}{}, update $\ell_1$, $r$, and $p$, since we may have $j>0$ which would mean $\ell_1$, $r$, and $p$ are not up-to-date.
    
    If $j>0$, we first update $\ell_1$ in the following way. If $w[j+1:n]=1$, we set $p\coloneq n$. If $\min(w[j+1:n])=1$, set $\ell_1\coloneq  \ell_1+j$; otherwise set $\ell_1\coloneq j$. Then we update $r$ using Corollary~\ref{corollary:sufInc} and $p$ using Corollary~\ref{corollary:psufInc} if $w[j+1:n]\neq 1$. 
    \item In the function \Call{Visit}{}, process $w$ as a Lyndon word if $r=n$, and as a necklace if $p=n$.
    \item The initial call is $\Call{Gen}{n,0,0,0,0,0,0}$ where the additional two zeroes at the end represent values for parameters $r$ and $p$.
    \item Change the call to $\Call{Gen}{}$ on Line 8 to $\Call{Gen}{n-1,c,1,c,1,1,1}$.
\end{enumerate}
See Appendix~\ref{appendix:neck} for a C program that implements these algorithms to generate quasinecklaces, necklaces, and Lyndon words.

\subsection{Analysis} \label{sec:analysis}

Recall that each node in the recursion tree for \Call{Gen}{} represents a suffix of some quasinecklace in $\Qs_k(n)$, or equivalently a recursive call to \Call{Gen}{}. We have already proved that the total number of nodes in the recursion tree of \Call{Gen}{} is proportional to $\Qn_k(n)$, which is itself proportional to $\Nn_k(n)$ by Theorem~\ref{theorem:QNeck}. 
From Remark~\ref{count:lyn}, $\Ln_k(n)$ is $\Theta(\Nn_k(n))$.
Additionally, apart from the calls to \Call{CMP}{}, the work done by each recursive call is constant. Therefore, to prove that our algorithm for generating necklaces or Lyndon words in colex order runs constant-amortized time, it is sufficient to show that the total amount of work done by all calls to \Call{CMP}{} is proportional to $\Nn_k(n)$, or equivalently, $\Qn_k(n)$.

We analyze the work done by all calls to $\Call{CMP}{p,\ell_1,q,\ell}$. Let $w$ be a length-$n$ word.  Recall that $\Call{CMP}{}$ returns the result of the comparison between $w[p:n-q+p]$ and $w[q:n]$ given that the preconditions outlined before Algorithm~\ref{algo:QN} are satisfied. In particular, $w[q:n]$ is the lexicographically smallest suffix of $w[2:n]$ and $p<q$.  Since the total number of nodes in the recursion tree for $\Call{Gen}{}$ is proportional to $\Qn_k(n)$, any constant amount of work done by a call to $\Call{CMP}{}$ is, when summed over all nodes, is proportional to $\Qn_k(n)$. Therefore, in our analysis of $\Call{CMP}{}$ we omit Lines 2 through 4, the first iteration of the \textbf{for} loop, and the last iteration of the \textbf{for} loop. The remaining iterations of the \textbf{for} loop compare $w[p+\ell+1:p+i]$ and $w[q+\ell+1:q+i]$ symbol-by-symbol for some integer $i$ such that $\ell<i<\min(n-q,q-p)$. We create a mapping that maps at most two comparisons to a single element of $\Qs_k(n-p+1)$.  For simplicity, and without loss of generality, we assume $p=1$. 

First, we define the domain of our mapping. We describe the words $w$ that are processed by the \textbf{for} loop when the index variable $i$ is such that $\ell < i < \min(n-q,q-1)$. 
Since $i$ is not the last iteration of the \textbf{for} loop, the conditions in the first $i-\ell+1$ iterations all evaluated to false. Additionally, all the conditions on Lines 2 through 4 evaluated to false. Therefore, we have  $w[1:i+1]=w[q:q+i]$ and furthermore $w[1:\ell+1]=a^\ell b$ and $w[q:q+\ell]=a^\ell b$,
where $a=\min(w)$ and $a < b$.
Since $w[q:n]$ is the lexicographically smallest suffix of $w[2:n]$, $w$ is a quasinecklace.
We explicitly define the domain in the following way. Let $\mathbf{D}(n)$ denote the set of all pairs $(w,i)$ of quasinecklaces $w \in \Qs_k(n)$ and integers $i$ such that:
\begin{enumerate}

    \item $w[1:\ell+1]=w[q:q+\ell]=a^{\ell}b$  where $\ell\geq 1$, 
        \item $w[1:i+1]=w[q:q+i]$, where $\ell < i < \min(n-q,q-1)$, and
    \item $w[q:n]$ is the lexicographically smallest suffix of $w[2:n]$.
\end{enumerate}
Now we are ready to present our mapping.  We define $f:\mathbf{D}(n)\to \Qs_k(n)$ by 
\[f(w,i) = \begin{cases} 
      a^i w[i+1:n] , & \text{if $w[i+1] > a = \min(w)$;} \\
      a^i (a+1) w[i+2:n] , & \text{otherwise.} \\
   \end{cases}\]
The mapping $f$ is constructed to map quasinecklaces to quasinecklaces.  Since we know that the number of quasinecklaces is proportional to the number of necklaces, it is sufficient to show that the size of the preimage of every quasinecklace is bounded above by a constant. We do this by treating both cases as separate functions with their own restricted domains. If the size of the preimage of every quasinecklace is bounded above by a constant for both cases, then the same is true for the mapping~$f$.

\begin{lemma}\label{lemma:firstComp}
   Let $\mathbf{D}_1 =\{(w,i)\in \mathbf{D}(n) \mid w[i+1] > \min(w)\}$. Then $f(w,i)$, restricted to $(w,i)\in \mathbf{D}_1$, is one-to-one. 
\end{lemma}
\begin{proof}
    Suppose $f(u,i)=f(v,j)$ where $(u,i),(v,j)\in \mathbf{D}_1$. Since $f(u,i)=f(v,j)$, we have $\min(u)=\min(v)=a$. But this means that $i=j$ since $f(u,i)$ begins with $a^jv[j+1]$ where $v[j+1]>a$. Now we have $a^iu[i+1:n]=a^iv[i+1:n]$, which implies that the smallest suffix of $u[2:n]$ is equal to the smallest suffix of $v[2:n]$. Then $u[1:i]=v[1:i]$, and thus $u=v$. Therefore, $f$ is one-to-one.
\end{proof}

\noindent
Note the second case in the definition of $f$ implies that $w[i+1] = \min(w)$.
\begin{lemma}\label{lemma:secondComp}
   Let $\mathbf{D}_2=\{(w,i)\in \mathbf{D}(n) \mid w[i+1] = \min(w)\}$. Then $f(w,i)$, restricted to $(w,i)\in \mathbf{D}_2$, is one-to-one.
\end{lemma}
\begin{proof}
    Suppose $f(u,i)=f(v,j)$ where $(u,i),(v,j)\in \mathbf{D}_2$. Since $f(u,i)=f(v,j)$, we have $\min(u)=\min(v)=a$. But this means that $i=j$ since $f(u,i)$ begins with $a^j(a+1)$. Now we have $a^i(a+1)u[i+2:n]=a^i(a+1)v[i+2:n]$, which implies that the smallest suffix of $u[2:n]$ is equal to the smallest suffix of $v[2:n]$. Then $u[1:i+1]=v[1:i+1]$, and thus $u=v$. Therefore, $f$ is one-to-one.
\end{proof}

Lemmas~\ref{lemma:firstComp} and \ref{lemma:secondComp} show that, when restricted to their respective domains, the first and second cases of $f$ are each one-to-one. Consequently, each of these cases could, in principle, map different elements of their domains to the same quasinecklace without violating injectivity within the case itself. 
Taking all cases together, we have that $f$ can map at most two distinct elements of $\mathbf{D}(t)$ to a single quasinecklace in $\Qs_k(t)$. Since the elements of $\mathbf{D}(t)$ correspond to all but a constant number of comparison operations performed by \Call{CMP}{}, each quasinecklace is therefore charged for at most two comparisons.  Therefore, since 
\begin{itemize}
    \item $\Qn_k(t) \in \mathcal{O}(\Nn_k(t))$ by Theorem~\ref{theorem:QNeck},
    \item $\sum\limits_{t=1}^n\Nn_k(t) \in \mathcal{O}(\Nn_k(n))$ by Remark~\ref{count:sum}, and
    \item $\Ln_k(n) \in \Theta(\Nn_k(n)$, 
\end{itemize}
the total number of comparison operations performed by \Call{CMP}{} across all recursive calls, considering all suffix lengths $1 \leq t \leq n$, is $\mathcal{O}( \Ln_k(n))$. 

\begin{theorem}\label{theorem:analysisN}
The necklaces $\Ns_k(n)$ and Lyndon words $\Ls_k(n)$ can each be listed in colex order in constant amortized time using $\mathcal{O}(n)$ space.
\end{theorem}


\subsection{Weight constraint} \label{subsection:weight}

The \emph{weight} of a word is the sum of its symbols.  To generate the necklaces $\Ns_k(n)$ or Lyndon words $\Ls_k(n)$ with weight at most $w$ in colex order, we need only make the following additional modifications to those outlined in Section~\ref{sec:modify} that modified Algorithm~\ref{algo:QN}:

\begin{enumerate}
    \item Add a parameter that maintains the  weight of the current suffix $w[j+1:n]$ that is updated appropriately at each recursive call.
    \item The function $\Next$ is modified so that it returns the maximum value that will not violate the weight constraint, accounting for the minimum prefix $1^{j-1}$.
\end{enumerate}

\begin{corollary}
The necklaces $\Ns_k(n)$ and the Lyndon words $\Ls_k(n)$ with weight at most $w$ can each be listed in colex order in constant amortized time using $\mathcal{O}(n)$ space.
\end{corollary}
\begin{proof}
The weight of the current suffix can easily be maintained at each recursive call with constant  work.  Similarly, the second modification also requires an extra constant amount of work to the function $\Next$.  In the proof of Theorem~\ref{theorem:QNeck}, observe that the one-to-one mapping of non-necklaces to necklaces preserves the weight constraint.  Furthermore, in the analysis from Section~\ref{sec:analysis}, they key function $f$ maps a suffix and an index to a quasinecklace with the same or smaller weight.  Thus, the same analysis applies to this bounded weight variant.
\end{proof}

\section{Application: de Bruijn sequences}\label{section:app}

A \emph{de Bruijn sequence} is a cyclic sequence of length $k^n$ that contains every length $n$-word over $\Sigma_k$ as a subword exactly once.  The \emph{Grandmama} de Bruijn sequence~\cite{Dragon&etal:2018}, can be constructed by concatenating together the primitive roots of the necklaces in $\Ns_k(n)$ as they appear in colex order.  For example, the colex listing of $\Ns_3(3)$ is:

\[ 111, 112, 122, 222, 132, 113, 123, 223, 133, 233, 333. \]
By concatenating together the primitive roots of each necklace, we obtain the Grandmama de Bruijn sequence:
\[1\cdot112\cdot122\cdot2\cdot132\cdot113\cdot123\cdot223\cdot133\cdot233\cdot3 \]
of length $3^3 = 27$, where $\cdot$ denotes concatenation. By applying Theorem~\ref{theorem:analysisN}, we obtain the following result.
\begin{corollary} \label{cor:Grandmama}
    The Grandmama de Bruijn sequence for $n,k\geq 1$ can be constructed in constant amortized time per symbol using $\mathcal{O}(n)$ space.
\end{corollary}
\noindent
This is an improvement by a factor of $n$ over the bound established in~\cite{Dragon&etal:2018}.  For $k=2$, the result was previously established in~\cite{Sawada&Williams&Wong:2017}. 

Recently, it has been demonstrated that concatenating the primitive roots of all necklaces in $\mathbf{N}_k(n)$ \emph{with weight at most $w$} as they appear in colex order yields a bounded weight de Bruijn sequence~\cite{Campbell&etal:2026}. Such sequences have found application in the construction of universal cycles for $k$-subsets and $k$-multisets~\cite{Campbell&etal:2026}. 
The result in Corollary~\ref{cor:Grandmama} also holds for this bounded-weight variant of the Grandmama de Bruijn sequence.

\section{Final remarks}\label{section:open}

In this paper we give a CAT algorithm to list the necklaces $\Ns_k(n)$ and Lyndon words $\Ls_k(n)$ (with an upper bound on their weight) as they appear in colex order.  When these words are listed in lexicographic order, efficient ranking and unranking algorithms are known~\cite{Kociumaka&Radoszewski&Rytter:2016}.
Thus, it is an interesting open question to efficiently rank/unrank necklaces and Lyndon words as they appear in colex order.   Such algorithms would likely lead to an efficient algorithm to decode the Grandmama de Bruijn sequence and its bounded-weight variant.

\section{Acknowledgment}
Joe Sawada (grant RGPIN-2025-03961) and Daniel Gabri\'{c} (grant RGPIN-2026-04863)  gratefully acknowledge research support from the Natural Sciences and Engineering Research Council of Canada (NSERC).

\bibliographystyle{new2}

\bibliography{abbrevs,simplest}

\appendix

\section{Table of Lyndon words, necklaces, and quasinecklaces}\label{appendix:table}

\begin{table}[H]
\centering
\caption{Lyndon words, necklaces, and quasinecklaces in colex order for $n=5$, $k=3$. Colour coding: blue (non-Lyndon), red (non-necklace). Read the lists top down, then left to right.}
\label{table:example}
\vspace{2mm}
\setlength{\extrarowheight}{4pt}   
\begin{tabular}{|p{0.30\linewidth}|p{0.30\linewidth}|p{0.30\linewidth}@{}|}
\hline
\multicolumn{1}{|c|}{\textbf{Lyndon words $\Ls_3(5)$}} & 
\multicolumn{1}{c|}{\textbf{Necklaces} $\Ns_3(5)$} & 
\multicolumn{1}{c|}{\textbf{Quasinecklaces} $\Qs_3(5)$} \\
\hline
\begin{minipage}[t]{\linewidth}
\vspace{-0.2cm}
\centering
\begin{multicols}{2}
\centering
11112 \\ 11212 \\ 11312 \\ 11122 \\ 12122 \\ 11222 \\ 12222 \\ 13222 \\ 11322 \\ 12322 \\ 13322 \\ 11132 \\ 12132 \\ 13132 \\ 11232 \\ 12232 \\ 13232 \\ 11332 \\ 12332 \\ 13332 \\ 11113 \\ 11213 \\ 12213 \\ 11313 \\ 12313 \\ 11123 \\ 12123 \\ 11223 \\ 12223 \\ 22223 \\ 13223 \\ 11323 \\ 12323 \\ 22323 \\ 13323 \\ 11133 \\ 12133 \\ 13133 \\ 11233 \\ 12233 \\ 22233 \\ 13233 \\ 23233 \\ 11333 \\ 12333 \\ 22333 \\ 13333 \\ 23333
\end{multicols}
\end{minipage}
&
\begin{minipage}[t]{\linewidth}
\vspace{-0.2cm}
\centering
\begin{multicols}{2}
\centering
\blue{11111} \\ 11112 \\ 11212 \\ 11312 \\ 11122 \\ 12122 \\ 11222 \\ 12222 \\ \blue{22222} \\ 13222 \\ 11322 \\ 12322 \\ 13322 \\ 11132 \\ 12132 \\ 13132 \\ 11232 \\ 12232 \\ 13232 \\ 11332 \\ 12332 \\ 13332 \\ 11113 \\ 11213 \\ 12213 \\ 11313 \\ 12313 \\ 11123 \\ 12123 \\ 11223 \\ 12223 \\ 22223 \\ 13223 \\ 11323 \\ 12323 \\ 22323 \\ 13323 \\ 11133 \\ 12133 \\ 13133 \\ 11233 \\ 12233 \\ 22233 \\ 13233 \\ 23233 \\ 11333 \\ 12333 \\ 22333 \\ 13333 \\ 23333 \\ \blue{33333}
\end{multicols}
\end{minipage}
&
\begin{minipage}[t]{\linewidth}
\vspace{-0.2cm}
\centering
\begin{multicols}{2}
\centering
\blue{11111} \\ 11112 \\ 11212 \\ \red{12212} \\ 11312 \\ \red{12312} \\ 11122 \\ 12122 \\ 11222 \\ 12222 \\ \blue{22222} \\ 13222 \\ 11322 \\ 12322 \\ 13322 \\ 11132 \\ 12132 \\ 13132 \\ 11232 \\ 12232 \\ 13232 \\ 11332 \\ 12332 \\ 13332 \\ 11113 \\ 11213 \\ 12213 \\ \red{13213} \\ 11313 \\ 12313 \\ \red{13313} \\ 11123 \\ 12123 \\ 11223 \\ 12223 \\ 22223 \\ 13223 \\ 11323 \\ 12323 \\ 22323 \\ 13323 \\ \red{23323} \\ 11133 \\ 12133 \\ 13133 \\ 11233 \\ 12233 \\ 22233 \\ 13233 \\ 23233 \\ 11333 \\ 12333 \\ 22333 \\ 13333 \\ 23333 \\ \blue{33333}
\end{multicols}
\end{minipage}
\vspace{0.2cm}
\\
\hline
\end{tabular}
\end{table}

\section{Table of pseudonecklaces and quasinecklaces}\label{appendix:pseudo}

  \begin{table}[H]
\centering
\caption{Pseudonecklaces and quasinecklaces in colex order for $n=8$, $k=2$. Quasinecklaces in red are not pseudonecklaces. Read the lists top down, then left to right.}
\label{table:pseudo-quasi}
\vspace{2mm}
\setlength{\extrarowheight}{4pt}
\begin{tabular}{|p{0.48\linewidth}|p{0.48\linewidth}|}
\hline
\multicolumn{1}{|c|}{\textbf{Pseudonecklaces} $\Ps(8)$} & 
\multicolumn{1}{c|}{\textbf{Quasinecklaces} $\Qs_2(8)$} \\
\hline
\begin{minipage}[t]{\linewidth}
\vspace{-0.2cm}
\centering
\begin{multicols}{2}
\centering
11111111 \\ 11111112 \\ 11121112 \\ 11112112 \\ 11212112 \\ 11122112  \\ 11111212 \\ 11211212 \\ 11121212 \\ 12121212 \\ 11221212  \\ 11112212 \\ 11212212  \\ 11122212 \\ 12122212 \\ 11222212 \\ 11111122 \\ 11121122 \\ 11221122 \\ 11112122 \\ 11212122 \\ 11122122 \\ 12122122 \\ 11222122 \\ 11111222 \\ 11211222 \\ 11121222 \\ 12121222 \\ 11221222 \\ 12221222 \\ 11112222 \\ 11212222 \\ 12212222 \\ 11122222 \\ 12122222 \\ 11222222 \\ 12222222 \\ 22222222
\end{multicols}
\end{minipage}
&
\begin{minipage}[t]{\linewidth}
\vspace{-0.2cm}
\centering
\begin{multicols}{2}
\centering
11111111 \\ 11111112 \\ 11121112 \\ 11112112 \\ 11212112 \\ 11122112 \\ \red{11222112} \\ 11111212 \\ 11211212 \\ 11121212 \\ 12121212 \\ 11221212 \\ \red{12221212} \\ 11112212 \\ 11212212 \\ \red{12212212} \\ 11122212 \\ 12122212 \\ 11222212 \\ \red{12222212} \\ 11111122 \\ 11121122 \\ 11221122 \\ 11112122 \\ 11212122 \\ \red{12212122} \\ 11122122 \\ 12122122 \\ 11222122 \\ \red{12222122} \\ 11111222 \\ 11211222 \\ 11121222 \\ 12121222 \\ 11221222 \\ 12221222 \\ 11112222 \\ 11212222 \\ 12212222 \\ 11122222 \\ 12122222 \\ 11222222 \\ 12222222 \\ 22222222
\end{multicols}
\end{minipage} 
\vspace{0.2cm}
\\
\hline
\end{tabular}
\end{table}

\newpage
\section{Generation Algorithm}\label{appendix:neck}
{\ssmall
\begin{code}
//---------------------------------------------------------------------------------
// NECKLACES, LYNDON WORDS, OR QUASINECKLACES IN COLEX ORDER IN O(1)-AMORTIZED TIME
//---------------------------------------------------------------------------------
#include <stdio.h>
#define MAX 100

int type,N,k,maxW,w[MAX],NECK=0,LYNDON=0,DB=0;
long long int total=0;

//---------------------------------------
// RETURNS -1 IF w[p:n-q+p] < w[q:n],
//          0 IF w[p:n-q+p] = w[q:n], and
//          1 IF w[p:n-q+p] > w[q:n].
//---------------------------------------
int CMP(int p, int l1, int q, int l) {
    int i;
    if (q == N) return 0;
    if (l1 < l) return 1;
    if (l1 > l) return -1;
    for (i=l; i<=N-q; i++) {
        if (p+i == q) return -1;
        if (w[p+i] > w[q+i]) return 1;
        else if (w[p+i] < w[q+i]) return -1;
    }
    return 0;
}
//--------------------------------------------------------
// COMPUTE suf(w[j:n])
//--------------------------------------------------------
int Suf(int j, int r, int cmp) {
    if (cmp == -1) return N-j+1;
    return r;
}
//--------------------------------------------------
// COMPUTE PSuf(w[j:n], r)
//--------------------------------------------------
int PSuf(int j, int r, int p, int cmp) {
    if (cmp == -1) return N-j+1;
    if (cmp == 0 && N-j+1-p == r) return N-j+1;
    return p;
}
//-----------------------------------------------------------------------------------
// RETURNS LARGEST c SUCH THAT cw[j+1:n] IS THE SUFFIX OF SOME LENGTH-N QUASINECKLACE
//-----------------------------------------------------------------------------------
int MaxSymbol(int j, int a, int l, int b, int l1) {
    int b1 = (l1+1 <= N-j) ? w[j+l1+1] : a;
    if (j == 0 || w[N] == 1 && j < N) return 1;
    if (j == N || j > l+1 || (j > 1 && a > 1)) return k;
    if (j == l + 1 || l == N-1) return b;
    if (j == 1 && b1 <= b && l1+1 == l && w[N] != a) return a;
    if (j == 1 && l1+1 > l && w[N] != a) return a;
    return (a-1 >= 1) ? a-1 : 1;
}
//---------------------------------------------------
// PRINT EACH QUASINECK/NECK/LYN
//---------------------------------------------------
void Visit(int r, int p) {
    int i;
    //TEST IF A WORD IS A NECK/LYN
    if (NECK && p < N) return;
    if (LYNDON && r < N) return;
    //OUTPUT DB SEQUENCE OR QUASINECK/NECK/LYN
    if (DB) for (i=1; i<=r; i++) printf("%d", w[i]);
    else {
        total++;
        for (i=1; i<=N; i++) printf("%d", w[i]);
        printf("\n");
    }
}
\end{code}\newpage
\begin{code}
//------------------------------------------------------------------------------
// LIST NECKLACES (LYNDON WORDS/PSEUDO-NECKLACES) IN COLEX ORDER.
//------------------------------------------------------------------------------
void Gen(int j, int a, int l, int b, int l1, int r, int p, int wt) {
    int c,MS,cmp,nr,np,b1;
    MS = MaxSymbol(j,a,l,b,l1);
    b1 = (l1+1 <= N-j) ? w[j+l1+1] : a;
    if (wt-(MS+(j-1)) < 0) MS = wt-(j-1);
    if (MS == 1) {
        if (j == 0) Visit(r,p);
        else if (w[N] == 1) Visit(1,N);
        else if (a > 1) Visit(N,N);
        else {
            //UPDATE r AND p
            cmp = CMP(1,l1+j,N-r+1,l); 
            nr = Suf(1,r,cmp);
            np = PSuf(1,r,p,cmp);
            Visit(nr,np);
        }
    }else {
        for(c=1; c<=MS; ++c){
            w[j] = c;
            if (j == N) Gen(N-1,c,1,c,1,1,1,wt-c);
            else if (c < a) Gen(j-1,c,1,w[j+1],1,N-j+1,N-j+1,wt-c);
            else if (c == a && (l1+1 > l || l1+1 == l && b1  <= b)) {
                cmp = CMP(j,l1+1,N-r+1,l);
                nr = Suf(j,r,cmp);
                np = PSuf(j,r,p,cmp);
                Gen(j-1,a,l1+1,b1,l1+1,nr,np,wt-c);
            } else if(c == a) Gen(j-1,a,l,b,l1+1,r,p,wt-c);
            else Gen(j-1,a,l,b,0,r,p,wt-c);
        }
        w[j] = 1;
    }
}
//-------------------------------------------------------------
void Input() {
    int i;
    printf("      Colex order                          Bounded weight    \n");
    printf("  --------------------------------      -----------------    \n");
    printf("  1.  Necklaces                         5. Necklaces         \n");
    printf("  2.  Lyndon words                      6. Lyndon words      \n");
    printf("  3.  Quasinecklaces                    7. Quasinecklaces    \n");
    printf("  4.  Grandmama De Bruijn sequence\n");
    printf("\n  ENTER object #: ");    scanf("%d", &type);

    if (type == 1 || type == 4 || type == 5) NECK = 1;
    if (type == 2 || type == 6) LYNDON = 1;
    if (type == 4) DB = 1;

    printf("\n  ENTER length n: ");    scanf("%d", &N);
    printf("  ENTER k: ");             scanf("%d", &k);

    if(type >= 5) { 
        printf("  ENTER Weight (%d - %d): ",N,k*N); scanf("%d", &maxW);
    } else maxW = k*N;
    
    printf("\n");
    
    for(i=1; i<=N; ++i) w[i] = 1;
}
//--------------------------------------------------
int main() {  
    Input();
    Gen(N,0,0,0,0,0,0,maxW);
    if(!DB) printf("Total = %lld\n",total);
    else printf("\n\n");
    return 0;
}
\end{code}}
\end{document}